\documentclass{amsart} 
\author{Masaki Kameko}
\address{Department of Mathematical Sciences,
Shibaura Institute of Technology,
307 Minuma-ku Fukasaku, Saitama-City 337-8570, Japan}
\email{kameko@shibaura-it.ac.jp}
\thanks{This work was supported by JSPS KAKENHI Grant Number JP17K05263.}
\keywords{classifying space, torsion, cohomology, gauge group}
\subjclass[2010]{55R40}

\usepackage{amssymb,amscd, pb-diagram}
\usepackage[initials]{amsrefs}
\newtheorem{theorem}{Theorem}[section]
\newtheorem{proposition}[theorem]{Proposition}
\newtheorem{lemma}[theorem]{Lemma}
\newtheorem{corollary}[theorem]{Corollary}

\theoremstyle{definition}

\BibSpec{arXiv}{%
  +{}{\PrintAuthors}{author}
  +{,}{ \textit}{title}
  +{}{ \parenthesize}{date}
  +{,}{ arXiv:}{eprint}
}

\begin{document}
\title{Torsion in classifying spaces of gauge groups}
\maketitle

\begin{abstract}
We determine when the integral homology of the classifying space of a $PU(n)$-gauge group over the sphere $S^2$ has torsion.
\end{abstract}


\section{Introduction} \label{section:1}


For a topological space $X$, we say $X$ has torsion if its integral homology does. Let $G$ be a compact connected Lie group. The cohomology of the connected Lie group $G$, its loop space $\Omega G$, and its classifying space $BG$ has been studied by many mathematicians after the pioneering works of Hopf, Bott, and Borel. The loop space $\Omega G$ has no torsion. The classifying space $BG$ has torsion if and only if $G$ does. 

Let $P\to X$ be a principal $G$-bundle over a paracompact space $X$. Then, there is a classifying map $f\colon X\to BG$. The group of bundle automorphisms covering the identity on X is called the gauge group $\mathcal{G}(P)$. The classifying space $B\mathcal{G}(P)$ is homotopy equivalent to the path-component of the mapping space $\mathrm{Map}(X, BG)$ containing the classifying map $f$ as in \cite{atiyah-bott-1983}, \cite{gottlieb-1972}. If $X=S^1$, since $\pi_1(BG)=\{0\}$, the mapping space $\mathrm{Map}(S^1, BG)$ is path-connected and it has torsion if and only if $G$ does. If $X=S^2$, since $\pi_2(BG)$ might not be zero, the mapping space $\mathrm{Map}(S^2, BG)$ may not be path-connected. The path-component that contains the trivial map is homotopy equivalent to the classifying space of the gauge group of the trivial $G$-bundle over $S^2$, and it has torsion if and only if $G$ does. However, the situation is different for other path-components that are homotopy equivalent to classifying spaces of gauge groups of nontrivial $G$-bundles. 

Let $SO(n)$ be the special orthogonal groups. Classification of $SO(n)$-bundles over $S^2$ is determined by the Stiefel-Whitney class $w_2\in \mathbb{Z}/2=\{0,1\}=\pi_2(BSO(n))$. The path-component of the mapping space is homotopy equivalent to the classifying space of the gauge group of the nontrivial $SO(n)$-bundle over $S^2$. Tsukuda \cite{tsukuda-1997} showed that it has no torsion for $n=3$. Minowa \cite{minowa-2023} proved that it has no torsion for $n=3,4$ and torsion for $n\geq 5$.

The special orthogonal group $SO(3)$ could be regarded as the projective unitary group $PU(2)=U(2)/S^1$. In this paper, we generalize Tsukuda's result for projective unitary groups $PU(n)$, $n\geq 2$ and determine when the classifying space of a $PU(n)$-gauge group over the sphere $S^2$ has torsion.



Throughout the rest of this paper, let $n$ be an integer greater than or equal to $2$. The second homotopy group $\pi_2(BPU(n))$ is isomorphic to the cyclic group $\mathbb{Z}/n$.  We identify the cyclic group $\mathbb{Z}/n$ with its complete set of representatives $\{ 0,1, \dots, n-1\}$. Let $k$ be an element in \[ \pi_2(BPU(n))=\mathbb{Z}/n=\{ 0,1, \dots, n-1\}. \] Let us denote by $\mathrm{Map}_k(S^2, BPU(n))$ the path-component of the mapping space $\mathrm{Map}(S^2, BPU(n))$ containing maps in the homotopy class $k$. Let $p$ be a prime number. Unless explicitly stated, $H^{*}(X)$ is the mod $p$ cohomology of the topological space $X$. The following is the $p$-local form of our result.

\begin{theorem} \label{theorem:1.1}
The following holds for $\mathrm{Map}_k (S^2, BPU(n))$.
\begin{itemize}
\item[{\rm (1)}] If $n\not \equiv 0 \mod (p)$, it has no $p$-torsion.
\item[{\rm (2)}] If $n\equiv 0 \mod (p)$  and $k\not \equiv 0\mod (p)$, it has no $p$-torsion.
\item[{\rm (3)}] If $n\equiv 0 \mod (p)$  and $k\equiv 0\mod (p)$, it has $p$-torsion.
\end{itemize}
\end{theorem}

As an immediate consequence of Theorem~\ref{theorem:1.1}, we obtain the following global form of our result.

\begin{corollary}\label{corollary:1.2} 
The topological space $\mathrm{Map}_k(S^2, BPU(n))$ has no torsion if and only if $k$ is relatively prime to $n$.
\end{corollary}

In particular, for $n\geq 2$, the topological space $\mathrm{Map}_1(S^2, BPU(n))$ has no torsion
even though the underlying Lie group $PU(n)$ has torsion.



This paper is organized as follows.
In Section~\ref{section:2}, we show the existence of $p$-torsion in $\mathrm{Map}_k(S^2, BPU(n))$ is equivalent to the triviality of certain induced homomorphism in the mod $p$ cohomology. Section~\ref{section:3} recalls the free double suspension in Takeda \cite{takeda-2021} and its elementary properties. Section~\ref{section:4} collects some elementary facts on the mod $p$ cohomology of $BU(n)$. In Section~\ref{section:5}, we prove Theorem~\ref{theorem:1.1} assuming Lemma~\ref{lemma:5.6} on an $n\times n$ matrix $B$.
In Section~\ref{section:6}, we prove Lemma~\ref{lemma:5.6}. 



The author would like to thank Yuki Minowa for his talk on \cite{minowa-2023} at the Homotopy Theory Symposium at the Osaka Metropolitan University on November 5, 2023. This work was inspired by his talk.


\section{Torsion}\label{section:2}


In this section, we show that the existence of $p$-torsion of a path-component is equivalent to the triviality of certain induced homomorphism.

Let us fix a fiber bundle $BU(n)\to BPU(n)$ induced by the obvious projection map $U(n)\to PU(n)$. We denote the inclusion map of its fiber by $\phi\colon BS^1 \to BU(n)$. It is a map induced by the obvious inclusion map $S^1 \to U(n)$ where $S^1$ consists of the scalar matrices in the unitary group $U(n)$. Consider the commutative diagram induced by the fiber bundle $BU(n)\to BPU(n)$.
\[
\begin{diagram}
\node{F_0} \arrow{e}\arrow{s} \node{\Omega^2_k BU(n)}\arrow{e,t} {\simeq} \arrow{s,r}{\iota_k}\node{\Omega^2_k BPU(n)} \arrow{s} \\
\node{F}\arrow{e,t}{\varphi}\arrow{s,l} {\simeq} \node{\mathrm{Map}_k(S^2, BU(n))} \arrow{e}\arrow{s,r}{\pi}  \node{\mathrm{Map}_k(S^2, BPU(n))} \arrow{s} \\
\node{BU(1)} \arrow{e,t}{\phi} \node{BU(n)} \arrow{e} \node{BPU(n)}
\end{diagram}
\]
Both vertical maps in the bottom-right square are evaluation maps at the base point of $S^2$, and
all maps in the bottom-right square are fibrations. Moreover, all horizontal and vertical sequences are fiber sequences. In particular, $\Omega_k^2 BU(n)$,  $\Omega_k^2 BPU(n)$ are fibers of evaluation maps.
Since 
\[
\Omega_k^2 BU(n)\to \Omega_k^2 BPU(n)
\]
is a homotopy equivalence, the fiber $F_0$ is contractible, and the map $F\to BS^1$ is also a homotopy equivalence.


The goal of this section is to prove the following proposition.

\begin{proposition}\label{proposition:2.1}
The following are equivalent.
\begin{itemize}
\item[{\rm (1)}] The topological space $\mathrm{Map}_k(S^2, BPU(n))$ has $p$-torsion.
\item[{\rm (2)}] The mod $p$ cohomology of $\mathrm{Map}_k(S^2, BPU(n))$ has a nonzero odd degree element.
\item[{\rm (3)}] The induced homomorphism $\varphi^{*}\colon H^2(\mathrm{Map}_k(S^2, BU(n)))\to 
H^2(F)$ is zero.
\end{itemize}
\end{proposition}


To establish the equivalence of (1) and (2) in Proposition~\ref{proposition:2.1}, we use the following lemma.

\begin{lemma}\label{lemma:2.2} Let $X$ be a topological space. Suppose that the integral homology groups $H_i(X;\mathbb{Z})$ are finitely generated abelian groups for all $i$, and the rational cohomology of $X$ has no nonzero odd degree element. Then, the mod $p$ cohomology $H^{*}(X;\mathbb{Z}/p)$ has a nonzero odd degree element if and only if $X$ has $p$-torsion. \end{lemma}

\begin{proof}
First, we prove that the assumptions of Lemma~\ref{lemma:2.2} imply that $H_{2j+1}(X;\mathbb{Z})$ is a finite abelian group for all $j$.
By the universal coefficient theorem, we have an isomorphism
\[
H^{2j+1}(X;\mathbb{Q})\simeq 
\mathrm{Ext}^1(H_{2j}(X;\mathbb{Z}), \mathbb{Q})
\oplus
\mathrm{Hom}(H_{2j+1}(X;\mathbb{Z}), \mathbb{Q}).
\]
By the assumption that the rational cohomology of $X$ has no nonzero odd degree element, we have
\[
 \mathrm{Hom}(H_{2j+1}(X;\mathbb{Z}), \mathbb{Q}) =\{0\}.
 \]
By the assumption that the integral homology groups $H_{i}(X;\mathbb{Z})$ are finitely generated, 
$H_{2j+1}(X;\mathbb{Z})$ is a finite abelian group.

Next, we show that if $X$ has $p$-torsion, then $H^{2j+1}(X;\mathbb{Z}/p)$ is nontrivial for some $j$.
By the universal coefficient theorem, we have an isomorphism
\[
H^{2j+1}(X;\mathbb{Z}/p) \simeq 
\mathrm{Ext}^{1}(H_{2j}(X;\mathbb{Z}), \mathbb{Z}/p) 
\oplus
\mathrm{Hom}(H_{2j+1}(X;\mathbb{Z}), \mathbb{Z}/p).
\]
If $X$ has $p$-torsion, $H_{2j+1}(X;\mathbb{Z})$ or $H_{2j}(X;\mathbb{Z})$ has $p$-torsion for some $j$.
Therefore, $H^{2j+1}(X;\mathbb{Z}/p)$ is nontrivial.
 
Finally, we show that if $H^{2j+1}(X;\mathbb{Z}/p)$ is nontrivial for some $j$, $X$ has $p$-torsion.
By the universal coefficient theorem, we have an isomorphism
\[
H^{2j+1}(X;\mathbb{Z}/p) \simeq 
\mathrm{Ext}^{1}(H_{2j}(X;\mathbb{Z}), \mathbb{Z}/p) 
\oplus
\mathrm{Hom}(H_{2j+1}(X;\mathbb{Z}), \mathbb{Z}/p).
\]
Suppose that 
\[
\mathrm{Hom}(H_{2j+1}(X;\mathbb{Z}), \mathbb{Z}/p)
\]
 is nontrivial. Then, since $H_{2j+1}(X;\mathbb{Z})$ is a finite abelian group, $H_{2j+1}(X;\mathbb{Z})$ has $p$-torsion. 
Suppose that  
\[
\mathrm{Ext}^{1}(H_{2j}(X;\mathbb{Z}), \mathbb{Z}/p)
\]
 is nontrivial. Then, since $H_{2j}(X;\mathbb{Z})$ is a finitely generated abelian group,  $H_{2j}(X;\mathbb{Z})$ has $p$-torsion.
Hence, in either case, $X$ has $p$-torsion.
\end{proof}

\begin{proof}[Proof of Proposition~\ref{proposition:2.1}, {\rm (1)} $\Leftrightarrow$ {\rm (2)}]
Let us consider the right vertical fiber sequence
\[
\Omega_k^2 BPU(n) \to \mathrm{Map}_k(S^2, BPU(n)) \to BPU(n)
\]
and Leray-Serre spectral sequences associated with this fiber sequence.
The $E_2$-page of the Leray-Serre spectral sequence for the integral homology
consists of finitely generated abelian groups, and so are the integral homology groups of $\mathrm{Map}_k(S^2, BPU(n))$.
The $E_2$-page of the Leray-Serre spectral sequence for the rational cohomology
has no nonzero odd degree element. So the rational cohomology of $\mathrm{Map}_k(S^2, BPU(n))$ also
has no nonzero odd degree element. Thus, by Lemma~\ref{lemma:2.2}, $\mathrm{Map}_k(S^2, BPU(n))$ has $p$-torsion if and only if its mod $p$ cohomology has a nonzero odd degree element.
\end{proof}



Let $c_i\in H^{2i}(BU(n))$ be the mod $p$ reduction of the $i^{\mathrm{th}}$ Chern class. The following proposition is what we need on the mod $p$ cohomology of $\mathrm{Map}_k(S^2, BU(n))$ in this section. Section~\ref{section:5} gives a more detailed description of the generator $x$ in terms of $c_2$ and the free double suspension we will define in Section~\ref{section:3}.

\begin{proposition}\label{proposition:2.3}
The following hold.
\begin{itemize}
\item[{\rm (1)}] 
$H^{*}(\mathrm{Map}_k(S^2, BU(n)))$ has no nonzero odd degree element.
\item[{\rm (2)}] As an abelian group, 
$
H^{2}(\mathrm{Map}_k(S^2, BU(n)))
$ is generated by two elements $\pi^*(c_1)$ and $x$ such that 
$\iota_k^*(x)\not=0$.
\end{itemize}
\end{proposition}

\begin{proof}
Consider the Leray-Serre spectral sequence associated with the middle vertical fiber sequence
\[
\Omega_k^2 BU(n) \to \mathrm{Map}_k(S^2, BU(n)) \to BU(n),
\]
converging to the mod $p$ cohomology of $\mathrm{Map}_k(S^2, BU(n))$. Then, the $E_2$-page has no nonzero odd degree element. Hence, the spectral sequence collapses at the $E_2$-page, and we obtain (1). Furthermore, we have 
\begin{align*}
E_\infty^{0,2}&=H^2(\Omega_k^2 BU(n))\simeq \mathbb{Z}/p,\\
E_\infty^{1,1}&=\{0\}, \\
E_\infty^{2,0}&=H^2(BU(n))=\mathbb{Z}/p\{ c_1\}.
\end{align*}
Hence, we have (2).
\end{proof}

\begin{proof}[Proof of Proposition~\ref{proposition:2.1},  {\rm (2)} $\Leftrightarrow$ {\rm (3)}]
We consider the Leray-Serre spectral sequence associated with the middle horizontal fiber sequence
\[
F \stackrel{\varphi}{\longrightarrow} \mathrm{Map}_k ( S^2, BU(n)) \longrightarrow \mathrm{Map}_k(S^2, BPU(n))
\]
converging to the mod $p$ cohomology of $ \mathrm{Map}_k ( S^2, BU(n))$.
The mod $p$ cohomology ring of $F\simeq BS^1$ is a polynomial ring generated by a single element $u$ of degree $2$. The $E_2$-page is given by
\[
E_2^{*,*}=H^{*}(\mathrm{Map}_{k}(S^2, BPU(n))) \otimes H^{*}(F).
\]

If the induced homomorphism
\[
\varphi^{*}\colon H^{2}(\mathrm{Map}_k(S^2, BU(n))) \to H^{2}(F)
\]
is nonzero, the induced homomorphism
\[
\varphi^{*}\colon H^{*}(\mathrm{Map}_k(S^2, BU(n))) \to H^{*}(F)
\]
is surjective. Then, by the Leray-Hirsh theorem, the induced homomorphism
\[
H^{*}(\mathrm{Map}_{k}(S^2, BPU(n))) \to H^{*}(\mathrm{Map}_k(S^2, BU(n))) 
\]
is injective and, by Proposition~\ref{proposition:2.3} (1),  the mod $p$ cohomology of $\mathrm{Map}_{k}(S^2, BPU(n))$ also has no nonzero odd degree element. 

If the induced homomorphism
\[
\varphi^{*}\colon H^{2}(\mathrm{Map}_k(S^2, BU(n))) \to H^{2}(F)
\]
is zero, $u$ does not survive to the $E_\infty$-page.
Hence, $d_2(u)\not=0$ or $d_3(u)\not=0$ must hold.
Relevant subgroups of $E_2$-page are as follows.
\begin{align*}
E_2^{0,2}&=\mathbb{Z}/p\{ u \}, \\
E_2^{1,1}&=\{0\}, \\
E_2^{2,1}&=\{0\},\\
E_2^{3,0}&=H^{3}(\mathrm{Map}_k(S^2, BPU(n))).
\end{align*}
Since $d_2(u)\in E_2^{2,1}=\{0\}$, we have $d_2(u)=0$. Therefore, $d_3(u)\not=0$.
Since $E_2^{1,1}=\{0\}$, the differential $d_2\colon E_2^{1,1}\to E_2^{3,0}$ is zero and we have $E_3^{3,0}=E_2^{3,0}$.
Since 
\[
d_3(u) \in E_3^{3,0}\simeq H^3(\mathrm{Map}_k(S^2, BPU(n)))
\]
is nonzero, the mod $p$ cohomology of $\mathrm{Map}_{k}(S^2, BPU(n))$ has the nonzero odd degree element $d_3(u)$.
\end{proof}


\section{Free double suspension}\label{section:3}

To describe the generator $x$ of $H^2(\mathrm{Map}_k(S^2, BU(n)))$ in Proposition~\ref{proposition:2.3} in more detail, we use the free double suspension 
\[
 \sigma\colon H^{*}(\mathrm{Map}_k(S^2, BU(n))) \to H^{*-2}(\mathrm{Map}_k(S^2, BU(n)))) 
\]
defined by Takeda in \cite{takeda-2021}. One may define the free double suspension over any coefficient groups. We focus on the mod $p$ cohomology. Our definition of $\sigma$ differs slightly from Takeda's $\hat{\sigma}_f^2$ in \cite{takeda-2021} but is the same homomorphism. 

In this section, let $X$ be a simply connected topological space. We denote by $*$ the base points of both $S^2$ and $X$. Let  $k$ be a homotopy class in $\pi_2(X)$ and $0$ is the homotopy class in $\pi_2(X)$ containing the trivial map.
Let 
\[
\mathrm{pr}_2\colon S^2 \times \mathrm{Map}_k(S^2, X)\to \mathrm{Map}_k(S^2, X)
\]
be the obvious projection map.
We use the evaluation maps
\[
\mathrm{ev} \colon S^2 \times \mathrm{Map}_k (S^2, X) \to X, \quad \mathrm{ev}(s, g)=g(s),
\]
and its restriction to $\mathrm{Map}_k(S^2, X)=\{*\} \times \mathrm{Map}_k(S^2, X)$,
\[
\pi \colon \mathrm{Map}_k(S^2, X)\to X, \quad \pi(g)=g(*), 
\]
to define a homomorphism
\[
\sigma\colon H^{*}(X)\to H^{*-2}(\mathrm{Map}_k (S^2, X)).
\]
Let us fix a generator of $H^2(S^2)\cong \mathbb{Z}/p$ and we denote it by $u_2$. We define $\sigma$ by
\[
\mathrm{ev}^*(x)-(\pi \circ \mathrm{pr}_2)^* (x)=u_2 \otimes \sigma(x).
\]
Let $\Omega_k^2 X=\pi^{-1}(*)$ and denote the inclusion map by $\iota_k\colon \Omega_k^2 X\to \mathrm{Map}_k(S^2, X)$.
We define \[
\tilde{\sigma}_k\colon H^{*}(X)\to H^{*-2}(\Omega_k^2 X)
\]
by $\iota_k^* \circ \sigma$. Proposition~\ref{proposition:3.1} (1) below is nothing but a particular form of Proposition 2.1 in \cite{takeda-2021}.


\begin{proposition}\label{proposition:3.1}
The homomorphism $\sigma$ satisfies the following.
\begin{itemize}
\item[{\rm (1)}] $ \sigma(x \cdot y) =\sigma(x) \cdot \pi^{*}(y)+\pi^{*}(x) \cdot \sigma(y)$,
\item[{\rm (2)}] for a cohomology operation $\mathcal{O}$ of positive degree, $\sigma(\mathcal{O} x)=\mathcal{O}\sigma(x)$.
\end{itemize}
\end{proposition}

\begin{proof}
(1) Since 
\begin{align*}\mathrm{ev}^{*}(x) \cdot \mathrm{ev}^{*}(y)&=(u_2 \otimes \sigma(x) +1\otimes \pi^{*}(x))\cdot (u_2 \otimes \sigma(y) +1\otimes \pi^{*}(y))
\\
&= u_2 \otimes \sigma(x) \cdot 1\otimes \pi^{*}(y) + 1\otimes \pi^{*}(x) \cdot u_2 \otimes \sigma(y)+1\otimes \pi^{*}(x)\cdot 1\otimes \pi^{*}(y)
\\
&=u_2 \otimes (\sigma(x)\cdot \pi^{*}(y)+\pi^{*}(x) \cdot \sigma(y))+1 \otimes (\pi^{*}(x)\cdot \pi^{*}(y)),
\end{align*} 
Hence, we have
\begin{align*}
\mathrm{ev}^{*}(x\cdot y)-(\pi\circ \mathrm{pr}_2)^*(x\cdot y)=u_2 \otimes (\sigma(x)\cdot \pi^{*}(y)+\pi^{*}(x) \cdot \sigma(y)).
\end{align*}
 (2) is also clear from the naturality of cohomology operation.
 \begin{align*}
 \mathcal{O} (\mathrm{ev}^{*}(x)-(\pi\circ \mathrm{pr}_2)^{*}(x))&=\mathrm{ev}^{*}(\mathcal{O}x)-(\pi\circ \mathrm{pr}_2)^{*}(\mathcal{O}x)
 \\
 &=u_2 \otimes \sigma(\mathcal{O}x),\\
 \mathcal{O}(u_2\otimes \sigma(x))&=u_2 \otimes \mathcal{O}\sigma(x),
 \end{align*}
 since $\mathcal{O} u_2=0$.
 Hence, we have 
 \[
 \sigma(\mathcal{O}x)=\mathcal{O}\sigma(x). \qedhere
 \]
\end{proof}

Next, we describe the relation between $\Omega_k^2 X$ and $\Omega_0^2 X$.
Let $X_1\vee X_2$ be the subspace of $X_1 \times X_2$ defined by
\[
X_1\vee X_2:=\{ (x_1,x_2)\in X_1 \times X_2 \;|\; x_1=* \mbox{ or } x_2=* \}
.
\]
Let $\nu\colon S^2 \to S^2 \vee S^2$ be the pinch map collapsing the sphere's equator.
We use it to define the addition on $\pi_2(X)$. 
Let $f\colon S^2\to X$ be a map representing $k\in \pi_2(X)$ and $\mathrm{c}_f\colon \Omega_0^2 X\to \{ f\}$ the obvious constant map.
Using $f$, we define 
\[
\mu_f\colon \Omega_0^2 X\to \Omega_k^2 X
\]
by
\begin{align*}
\mu_f(g)(s)&=f(s_1) \quad \mbox{if $\nu(s)=(s_1,*)$,}\\
\mu_f(g)(s)&=g(s_2) \quad \mbox{if $\nu(s)=(*, s_2)$.}
\end{align*}
The following lemma is a weak form of  Lemma 2.2 in \cite{takeda-2021}.
We use it to prove Proposition~\ref{proposition:5.1}.


\begin{lemma}\label{lemma:3.2}
Let $x$ be an element in $H^{i}(X)$.
If $i\not=2$, then we have
\[
\mu_f^{*} \circ \tilde{\sigma}_k(x)=\tilde{\sigma}_0(x).
\]
\end{lemma}

\begin{proof}
We have the following commutative diagram by the definition of $\mu_f$.
\[
\begin{diagram}
\node{S^2\times \Omega_0^2X} \arrow{e,t}{ \nu\times 1} \arrow{s,l}{1\times \mu_f} \node{S^2 \times \Omega_0^2X \vee S^2 \times \Omega_0^2X} \arrow[2]{e,t}{1 \times \mathrm{c}_f \vee 1\times 1} \node[2]{S^2 \times\{ f\} \vee S^2 \times \Omega_0^2 X} \arrow{s,r}{\mathrm{ev} \vee \mathrm{ev}}
\\
\node{S^2 \times\Omega^2_k X} \arrow[3]{e,t}{\mathrm{ev}} \node[3]{X,}
\end{diagram}
\]
where we choose $f$ as the base point of both $ \{ f\}$ and $\Omega_k^2 X$, 
and the constant map $S^2\to \{*\}$ as the base point of $\Omega_0^2X$.
Since the reduced mod $p$ cohomology $\widetilde{H}^{i}(S^2\times \{f\})\simeq \widetilde{H}^{i}(S^2)$ is trivial for $i\not=2$, we  have
isomorphisms 
\[
H^{i}(S^2 \times \{f\} \vee S^2 \times \Omega_0^2 X) \to H^i(S^2 \times \Omega_0^2 X)
\]
and desired identity
\[
\tilde{\sigma}_0(x)=\mu_f^{*} \circ \tilde{\sigma}_k(x).
\]
 for $x\in H^i(X)$, $i\not=2$.
\end{proof}


\section{Cohomology of $BU(n)$} \label{section:4}

In this section, we collect some elementary properties of the mod $p$ cohomology ring of $BU(n)$ and the induced homomorphism
\[
\phi^{*}\colon H^{*}(BU(n))\to H^{*}(BS^1).
\]


Let us fix a generator $u$ of $H^2(BU(1))=H^2(BS^1)\simeq \mathbb{Z}/p$. 
Let 
\[
\iota\colon BU(1)^n \to BU(n)
\] be the map induced by the inclusion map of the maximal torus $U(1)^n$ consisting of diagonal matrices.
Let \[
B\mathrm{pr}_i\colon BU(1)^n \to BU(1) =BS^1
\]
be the map induced by the projection of $U(1)^n$ to its $i^{\mathrm{th}}$ factor $U(1)$, defined by $(x_1, \dots, x_n)\mapsto x_i$.
We denote $B\mathrm{pr}_i^{*}(u)\in H^2(BU(1)^n)$ by  $t_i$. The mod $p$ cohomology of $BU(1)^n$ is a polynomial ring generated by
$t_1, \dots, t_n$ and the induced homomorphism
\[
\iota^*\colon H^{*}(BU(n))\to H^{*}(BU(1)^n)=\mathbb{Z}/p[t_1, \dots, t_n]
\]
is injective, and its image is the set of symmetric polynomials in $t_1, \dots, t_n$.
In particular, $c_i$ is defined as the element such that
 $\iota^{*}(c_i)$ is the $i^{\mathrm{th}}$ elementary symmetric polynomial in $t_1, \dots, t_n$.
Let us define $s_i$ by
\[
\iota^{*}(s_i)=\sum_{j=1}^n t_j^i.
\]
The map $\phi\colon BS^1 \to BU(n)$ factors through
\[
BS^1\stackrel{\delta}{\longrightarrow} BU(1)^n \stackrel{\iota}{\longrightarrow} BU(n),
\]
where $\delta$ is the map induced by the diagonal map $x\mapsto (x,\dots, x)$.
Since $\delta^{*}(t_i)=u$ for $i=1, \dots, n$, we have 
\[
\phi^{*}(s_i)=n u^i
\]
  and \[ \phi^{*}(c_i)=\binom{n}{i} u^i.
  \]


We use the following Lemma~\ref{lemma:4.1} to prove Proposition~\ref{proposition:5.5}. 
The corresponding identity in symmetric polynomials is known as Newton's identity.

\begin{lemma} \label{lemma:4.1} In the mod $p$ cohomology of $BU(n)$, 
for $i\geq 0$, we have relations
\[
s_{n+i+1}+\sum_{j=1}^{n} (-1)^j c_j s_{n+i-j+1}=0.
\]
\end{lemma}

\begin{proof}
Let us define symmetric polynomials $h_{i+2,n-1}, \dots, h_{n+i,1}$. 
For $\ell=i+2,\dots, n+i$, let $h_{\ell, n+i+1-\ell}$ be the sum of monomials in the polynomial ring $\mathbb{Z}/p[t_1, \dots, t_n]$ obtained from $t_1^{\ell} t_2 \cdots t_{n+i+2-\ell}$ by permuting $1, \dots, n+j+2-\ell$ in $1, \dots, n$. Then, we have
\begin{align*}
\iota^{*}(c_1)\cdot \iota^{*}(s_{n+i})&=\iota^{*}(s_{n+i+1})+h_{n+i,1},\\
\iota^{*}(c_j)\cdot \iota^{*}(s_{n+i+1-j})&=h_{n+i+2-j, j-1} +h_{n+i+1-j,j}, 
\intertext{for $2\leq j \leq  n-1$ and}
\iota^{*}(c_n)\cdot \iota^{*}(s_{i+1})&=h_{i+2, n-1}.
\end{align*}
Therefore, we have 
\begin{align*}
& \iota^{*}(s_{n+i+1}+\sum_{j=1}^{n} (-1)^j c_i s_{n+i+1-j})
\\
=& \iota^{*}(s_{n+i+1})-(\iota^{*}(s_{n+i+1})+h_{n+i,1})+\sum_{j=2}^{n-1} (-1)^j \left(h_{n+i+2-j,j-1}+h_{n+i+1-j, j}\right)+ (-1)^{n} h_{i+2, n-1}
\\
=&0.
\end{align*}
Since $\iota^*$ is injective, it completes the proof.
\end{proof}


If $p$ is an odd prime, let 
\[
\wp^i\colon H^j(X)\to H^{j+2i(p-1)}(X)
\]
 be the $i^{\mathrm{th}}$ Steenrod reduced power.
If $p=2$, let $\wp^1=\mathrm{Sq}^2$ and $\wp^{2^{\ell-1}}= \mathrm{Sq}^{2^\ell}$ for $\ell\geq 2$, where 
\[
\mathrm{Sq}^i\colon H^{j}(X)\to H^{j+i}(X)
\] is the $i^{\mathrm{th}}$ Steenrod square. 
We define cohomology operations $\mathcal{Q}_\ell$ inductively by 
$\mathcal{Q}_1=\wp^1$, 
\[
\mathcal{Q}_\ell=\wp^{p^{\ell-1}} \mathcal{Q}_{\ell-1} - \mathcal{Q}_{\ell-1}\wp^{p^{\ell-1}} 
\]
for $\ell\geq 2$.
Cohomology operations $\mathcal{Q}_\ell$ have the following properties
\begin{itemize}
\item[{\rm (1)}] $\mathcal{Q}_{\ell}(x\cdot y)=\mathcal{Q}_{\ell}(x)\cdot y+x\cdot \mathcal{Q}_{\ell}( y)$ for $x, y\in H^{*}(BU(1)^n)$, 
\item[{\rm (2)}] $\mathcal{Q}_{\ell} t_i=t_i^{p^\ell}$ for $t_1, \dots, t_n$ in $H^2(BU(1)^n)$.
\end{itemize}
With these properties, we have the following Lemma~\ref{lemma:4.2}. We will use it to prove Proposition~\ref{proposition:5.2}.


\begin{lemma} \label{lemma:4.2}
In the mod $p$ cohomology of $BU(n)$, for $\ell\geq 1$, we have 
\[
\mathcal{Q}_\ell c_2=s_1s_{p^\ell}-s_{p^\ell+1}.
\]
\end{lemma}

\begin{proof} 
On the one hand, since 
\[
\iota^{*}(c_2)=\sum_{1\leq i<j\leq n} t_i t_j,
\]
by direct calculation, we have
\[
\iota^{*}(\mathcal{Q}_\ell (c_2))=\sum_{1\leq i<j\leq n} (t_i^{p^\ell}t_j+t_i t_j^{p^\ell}).
\]
On the other hand, we have
\[
\iota^{*}(s_{p^{\ell}} s_1-s_{p^\ell+1})=(\sum_{i=1}^n t_i^{p^{\ell}}) (\sum_{j=1}^n t_j)-\sum_{i=1}^n t_i^{p^\ell+1}=
\sum_{1\leq i<j\leq n} (t_i^{p^{\ell}}t_j+t_i t_j^{p^\ell}).
\]
Hence, we obtain the desired identity.
\end{proof}


\section{Proof of Theorem~\ref{theorem:1.1}}\label{section:5}

In this section, we consider the commutative diagram
\[
\begin{diagram}
\node{F} \arrow{e,t}{\varphi} \arrow{s,l}{\simeq} \node{\mathrm{Map}_k(S^2, BU(n))} \arrow{s,r}{\pi} \\
\node{BS^1} \arrow{e,t}{\phi} \node{BU(n)}
\end{diagram}
\]
We begin with the following refinement of Proposition~\ref{proposition:2.3} (2).


\begin{proposition}\label{proposition:5.1}
As an abelian group,  
$
H^{2}(\mathrm{Map}_k(S^2, BU(n)))
$
 is generated by $\pi^*(c_1)$ and $\sigma(c_2)$.
\end{proposition}

\begin{proof}
Let $\lambda \colon BSU(n)\to BU(n)$ and $\lambda'\colon \Omega^2 BSU(n) \to \Omega^2_0 BU(n)$ be maps induced by the inclusion map $SU(n)\to U(n)$.
We have the following commutative diagram by Lemma~\ref{lemma:2.2} and the naturality of cohomology suspension.
\[
\begin{diagram}
\node{H^{2}(\Omega^2 BSU(n))} \node{H^{4}(BSU(n))}\arrow{w,t}{\tilde{\sigma}} \\
\node{H^{2}(\Omega_0^2 BU(n))} \arrow{n,l}{\lambda'^{*}} \node{H^{4}(BU(n))}\arrow{w,t}{\tilde{\sigma}_0} \arrow{n,r}{\lambda^*}\\
\node{H^{2}(\Omega_k^2 BU(n))} \arrow{n,l}{\mu_f^*} \node{H^{4}(BU(n))}\arrow{w,t}{\tilde{\sigma}_k} \arrow{n,r}{=}
\end{diagram}
\]
The top horizontal homomorphism $\tilde{\sigma}$ is the composition of cohomology suspensions 
\[
H^{4}(BSU(n))\to H^{3}(\Omega BSU(n))\to H^{2}(\Omega^2 BSU(n))
\]
and it is an isomorphism. Since $H^4(BSU(n))=\mathbb{Z}/p$ is generated by $\lambda^{*}(c_2)$, we have
\[
\lambda'^* \circ \mu_f^* \circ \tilde{\sigma}_k (c_2)=\tilde{\sigma}\circ \lambda^*(c_2)\not=0.
\]
Therefore, we obtain
\[
\tilde{\sigma}_k(c_2)=\iota_k^*\circ \sigma(c_2)\not=0. 
\]
By Proposition~\ref{proposition:2.3} (2), $\pi^{*}(c_1)$ and $\sigma(c_2)$ generate 
$H^{2}(\mathrm{Map}_k(S^2, BU(n)))$.
\end{proof}

Let $u\in H^2(F)=H^2(BS^1)\simeq \mathbb{Z}/p$ be the generator fixed in Section~\ref{section:4}.
Let us define $\alpha_i, \beta \in \mathbb{Z}/p$ by
\begin{align*}
\alpha_i u^i&=\varphi^{*}\circ \sigma(s_{i+1}), \\
\beta u&= \varphi^{*}\circ \sigma(c_2).
\end{align*}



\begin{proposition}\label{proposition:5.2}
If $n\equiv 0 \mod (p)$,  we have
$
\beta =-\alpha_{p^{\ell}}
$
for $\ell\geq 1$. 
\end{proposition}

\begin{proof}
On the one hand, by the definition of $\beta$, we have
\[
\varphi^*\circ \sigma( c_2)=\beta u.
\]
Applying $\mathcal{Q}_\ell$, we have 
\[
\varphi^*\circ \sigma( \mathcal{Q}_\ell c_2)=(\beta u)^{p^{\ell}}=\beta u^{p^{\ell}}.
\]
On the other hand, 
by Lemma~\ref{lemma:4.2}, in the mod $p$ cohomology of $BU(n)$, we have the relation
\[
\mathcal{Q}_\ell c_2=s_1 s_{p^{\ell}}- s_{p^\ell+1}.
\]
Applying $\varphi^{*}\circ \sigma$, we have
\begin{align*}
\varphi^{*}\circ \sigma (\mathcal{Q}_\ell c_2)&=\varphi^{*}\circ \sigma(s_1) \cdot \phi^{*}(s_{p^{\ell}}) +\phi^{*}(s_1) \cdot \varphi^{*}\circ \sigma(s_{p^\ell})- \varphi^*\circ \sigma(s_{p^\ell+1})
\\
&=n \alpha_1 u^{p^\ell}  +n \alpha_{p^{\ell-1}} u^{p^\ell}-\alpha_{p^{\ell}} u^{p^\ell}
\\
&=-\alpha_{p^{\ell}} u^{p^{\ell}}.
\end{align*}
Hence, we have $\beta=-\alpha_{p^{\ell}}.$
\end{proof}



Summing up Propositions~\ref{proposition:5.1} and \ref{proposition:5.2},
we have the following Proposition~\ref{proposition:5.3}. 
It reduces the proof of Theorem~\ref{theorem:1.1} to
the computation of $\alpha_p$.

\begin{proposition}\label{proposition:5.3}
The following are equivalent.
\begin{itemize}
\item[{\rm (1)}] $\varphi^{*}\colon H^{2}(\mathrm{Map}_k(S^2, BU(n)))\to H^{2}(F)$ is zero,
\item[{\rm (2)}] $\phi^{*}(c_1)=0$ and $\beta=0$,
\item[{\rm (3)}]  $\phi^{*}(c_1)=0$ and $\alpha_p=0$.
\end{itemize}
\end{proposition}

\begin{proof}
Since $H^{2}(\mathrm{Map}_k(S^2, BU(n)))$ is generated by $\pi^{*}(c_1)$ and $\sigma(c_2)$, 
(1) and (2) are equivalent. Under the assumption that $\phi^*(c_1)=0$, we have $n\equiv 0 \mod (p)$.
Then, by Proposition~\ref{proposition:5.2}, we have
\[
\beta=-\alpha_p. 
\]Hence, (2) and (3) are equivalent.
\end{proof}



By computing $\alpha_p$, we complete the proof of Theorem~\ref{theorem:1.1}.

\begin{proposition} \label{proposition:5.4}
We have $\alpha_0=k$.
\end{proposition}

\begin{proof}
Let $f\colon S^2\to BU(n) \in \mathrm{Map}_k(S^2, BU(n))$. By definition, 
we have 
\[
f^{*}(c_1)=k u_2.
\]
Let 
\[
\mathrm{i}_f\colon S^2 \to S^2 \times \mathrm{Map}_k(S^2, BU(n))
\]
be a map defined by $t\mapsto (t,f)$.
Then, we have
\[
f= \mathrm{ev}\circ \mathrm{i}_f
\]
and 
\[
\pi \circ \mathrm{pr}_2 \circ \mathrm{i}_f
\]
is a constant map $S^2\to \{ f(*) \}$.
It implies that 
\[
\mathrm{i}_f^*(\mathrm{ev}^{*}(c_1)-(\pi \circ \mathrm{pr}_2)^{*}(c_1))=f^{*}(c_1)=k u_2.
\]
When we restrict  $\mathrm{i}_f^*$ to $H^{2}((S^2, *) \times \mathrm{Map}_k(S^2, BU(n)))$, it is injective.
So, we have
\[
\mathrm{ev}^{*}(c_1)-(\pi \circ \mathrm{pr}_2)^{*}(c_1)=k u_2\otimes 1.
\]
Hence, by the definition of $\sigma$, we have $\sigma(c_1)=k$.
\end{proof}



\begin{proposition} \label{proposition:5.5} 
If $n\equiv 0 \mod (p)$, we have $\alpha_{p}=k$. 
\end{proposition}

We use the following Lemma~\ref{lemma:5.6} to prove Proposition~\ref{proposition:5.5}. We will prove it in the next section. Let $B$ be an $n\times n$ matrix whose $(i,j)$-entry is given by integers
\[
b_{1,j}=(-1)^{j+1} \binom{n}{j}
\]
for $1\leq j\leq n$ and 
$
b_{i,j}=1
$
if $i=j+1$, $b_{i,j}=0$ if $i\not=j+1$
for $2\leq i \leq n$, $1\leq j \leq n$.

\begin{lemma} \label{lemma:5.6}
When we regard $B$ as an element in $SL_n(\mathbb{Z}/p)$, 
the order of $B$ is a power of $p$.
\end{lemma}

\begin{proof}[Proof of Proposition~\ref{proposition:5.5}]
By Lemma~\ref{lemma:4.1}, in $H^{*}(BU(n))$, we have 
\[
s_{n+i+1} + \sum_{j=1}^{n} (-1)^j c_j s_{n+i+1-j} =0
\]
for $i\geq 0$. Applying $\varphi^{*}\circ \sigma$, we have
\[
\alpha_{n+i} u^{n+i} + \sum_{j=1}^{n} (-1)^j \phi^{*}(c_j) \cdot \alpha_{n+i-j} u^{n+i-j} +\sum_{j=1}^{n} (-1)^j \varphi^* \circ \sigma (c_j) \cdot \phi^{*}(s_{n+i-j+1}) =0.
\]
Since $\phi^{*}(s_{n+i-j+1})=0$, we obtain
\[
\alpha_{n+i} u^{n+i} + \sum_{j=1}^{n} (-1)^j \phi^{*}(c_j) \cdot \alpha_{n+i-j} u^{n+i-j}  =0.
\]
 Furthermore, since 
$\displaystyle \phi^*(c_j)=\binom{n}{j} u^i$, we have
\[
\alpha_{n+i}  + \sum_{j=1}^{n}  (-1)^{j} \binom{n}{j} \alpha_{n+i-j} =0.
\]
Thus, we have
\begin{align*}
\alpha_{n+i}&=\sum_{j=1}^n (-1)^{j+1} \binom{n}{j} \alpha_{n+i-j}, \\
\alpha_{n-1+i}&=\alpha_{n-1+i}, \\
& \;\;\vdots \\
\alpha_{1+i}&=\alpha_{1+i}.
\end{align*}
Therefore, put these identities together in matrix form, using the $n\times n$ matrix $B$ that we just defined, we have
\[
\begin{pmatrix} \alpha_{n+i} \\ \vdots \\ \alpha_{1+i} \end{pmatrix}=B\begin{pmatrix} \alpha_{n-1+i} \\ \vdots \\ \alpha_{i} \end{pmatrix} =\cdots = B^{i+1}
\begin{pmatrix} \alpha_{n-1} \\ \vdots \\ \alpha_{0} \end{pmatrix},
\]
for $i\geq 0$.
By Lemma~\ref{lemma:5.6}, the order of $B$ as an element of $SL_n(\mathbb{Z}/p)$ is a power of $p$.
Hence, for some positive integer $\ell$, we have 
\[
\alpha_{p^\ell}=\alpha_0=k.
\]
By Proposition~\ref{proposition:5.2}, we have $\alpha_{p^\ell}=-\beta=\alpha_p$.
Therefore, we obtain $\alpha_p=k$.
\end{proof}


Proposition~\ref{proposition:5.7} below is immediate from Proposition~\ref{proposition:5.5}
and
it completes the proof of Theorem~\ref{theorem:1.1}.

\begin{proposition}\label{proposition:5.7}
The following holds.
\begin{itemize}
\item[{\rm (1)}]
If $n\not \equiv 0 \mod (p)$, then $\phi^{*}(c_1)\not=0$,
\item[{\rm (2)}]
If $n\equiv 0 \mod (p)$  and $k\not \equiv 0 \mod (p)$, then  
$\alpha_p \not=0$,
\item[{\rm (3)}] If $n\equiv 0 \mod (p)$ and $k\equiv 0 \mod (p)$, then
$\phi^{*}(c_1)=0$ and $\alpha_p=0$.
\end{itemize}
\end{proposition}


\section{Proof of Lemma~\ref{lemma:5.6}} \label{section:6}

In this section, we deal with unimodular $n\times n$ matrices. Unless otherwise clear from the context, matrix entries are integers. What we do in what follows is to find the transpose of the Jordan matrix similar to $B$ in Section~\ref{section:5}.


\begin{proposition}
\label{proposition:6.1}
There is a unimodular $n\times n$ matrix $A$ such that $A^{-1}BA=D$ where 
$(i,j)$-entry $d_{i,j}$ of $D$ is $d_{i,j}=1$ if $i=j$ or $i=j+1$ and $d_{i,j}=0$ if otherwise.
\end{proposition}

We prove this proposition by giving such an $A$ explicitly.  Before we do it, we complete the proof of Lemma~\ref{lemma:5.6}.

\begin{proof}[Proof of Lemma~\ref{lemma:5.6}]
By Proposition~\ref{proposition:6.1}, we have
\[
B=AD  A^{-1}.
\]
The matrix $D$ belongs to the subgroup $U_n$ of $SL_n(\mathbb{Z}/p)$ consisting of lower triangular matrices whose diagonal entries are $1$. The subgroup $U_n$ is a $p$-group.
Therefore, the order of $D$ is a power of $p$.
Hence, the order of $B$ is also the power of $p$.
\end{proof}

Now, we prove Proposition~\ref{proposition:6.1} by defining $A$ explicitly.

\begin{proof}[Proof of Proposition~\ref{proposition:6.1}]
Let $A$ be the $n\times n$ unimodular upper triangular matrix whose $(i,j)$-entry is given by 
\[
a_{i,j}=\binom{n-i}{n-j}.
\]
We show that $(i,j)$-entries of $BA$ and $AD$ are equal to $\displaystyle \binom{n-i+1}{n-j}$ for $1\leq i\leq n, 1\leq j\leq n$.

Recall that $B$ is the $n\times n$ unimodular matrix whose $(i,j)$-entry is given as follows:
For $i=1$, $1\leq j \leq n$, the $(1,j)$-entry of $B$ is given by
\[
b_{1,j}=(-1)^{j+1} \binom{n}{j}, 
\]
and, for $2\leq i \leq n$, $1\leq j \leq n$, the $(i,j)$-entry of $B$ is given by
\[
\begin{array}{rcll} 
b_{i,j}&=&1 & \mbox{if $i=j+1$,} \\
b_{i,j}&=&0 & \mbox{otherwise.}
\end{array}
\]

\noindent (1) 
For $1\leq j\leq n$, the $(1,j)$-entry of $BA$ is given by
\begin{align*}
\sum_{\ell=1}^n b_{1,\ell} a_{\ell,j}
&=\sum_{\ell=1}^j b_{1,\ell} a_{\ell,j}
\\
&=\sum_{\ell=1}^j (-1)^{\ell+1} \binom{n}{\ell}  \cdot \binom{n-\ell}{n-j}
\\
&=\sum_{\ell=1}^j (-1)^{\ell+1} \dfrac{n!}{(n-\ell)! \ell!} \cdot \dfrac{(n-\ell)!}{(n-j)!(j-\ell)!}
\\
&=\sum_{\ell=1}^j (-1)^{\ell+1} \dfrac{n!}{(n-j)! j!} \cdot \dfrac{j!}{(j-\ell)!\ell!}
\\
&=\sum_{\ell=1}^j (-1)^{\ell+1} \binom{n}{n-j} \binom{j}{\ell}
\\
&= \binom{n}{n-j} \left( \sum_{\ell=1}^j (-1)^{\ell+1}  \binom{j}{\ell}\right)
\\
&=\binom{n}{n-j}
\\
&=\binom{n-1+1}{n-j}
\end{align*}
For $2\leq i \leq n$, $1\leq j\leq n$, the $(i,j)$-entry of $BA$ is given by
\begin{align*}
\sum_{\ell=1}^n b_{i,\ell}a_{\ell,j}&=b_{i,i-1} a_{i-1,j}
\\
&=a_{i-1,j}
\\
&=\binom{n-i+1}{n-j}.
\end{align*}

\noindent (2) 
For $1\leq i \leq n$, $1\leq j \leq n$, the $(i,j)$-entry of $AD$ is given by
\begin{align*}
\sum_{\ell=1}^n a_{i,\ell}d_{\ell,j} &= a_{i,j}d_{j,j}+a_{i,j+1}d_{j+1,j}
\\
&=a_{i, j}+a_{i,j+1}
\\
&=\binom{n-i}{n-j}+\binom{n-i}{n-j-1}
\\
&=\binom{n-i+1}{n-j}.
\end{align*}
It completes the proof of Proposition~\ref{proposition:6.1}
\end{proof}




\end{document}